\documentclass[reqno,10pt]{amsart}
\usepackage{amsfonts,amssymb,latexsym,epsfig,amsmath,color,xspace}
\newtheorem{theorem}{Theorem}[section]

\newtheorem{corollary}[theorem]{Corollary}
\newtheorem{lemma}[theorem]{Lemma} %{Lemma}
\theoremstyle{remark}

\newtheorem{remark}[theorem]{Remark}
\numberwithin{equation}{section}
\renewenvironment{proof}[1][Proof]{\medskip\noindent\textit{#1 ---\;}}{\\ 
    \null\  \hfill\textbf{Q.E.D.}\medskip}

\newcommand{\ie}{\textit{i.e.}\xspace}

\newcommand{\R}{\mathbb{R}}
\newcommand{\hyp}[1]{$(\mathbf{H#1})$}
\newcommand{\nhyp}[1]{\noindent $(\mathbf{H#1})$}
\newcommand{\gyp}[1]{$(\mathbf{G#1})$}
\newcommand{\ds}{\,\mathrm{d}s}

\newcommand{\dsigma}{\,\mathrm{d}\sigma}
\begin{document}
    \title[On the regularizing effect for Hamilton-Jacobi equations]{On the
        regularizing effect for unbounded solutions of first-order
        Hamilton-Jacobi equations}

\author{Guy Barles}
\address{Guy Barles\newline
    Laboratoire de Math\'ematiques et Physique Th\'eorique (UMR CNRS 7350)\newline
    F\'ed\'eration Denis Poisson (FR CNRS 2964)\newline
    Universit\'e F. Rabelais - Tours\newline
    Parc de Grandmont\newline
    37200 Tours, France}
\email{guy.barles@lmpt.univ-tours.fr}
\urladdr{http://www.lmpt.univ-tours.fr/$\sim$barles}

\author{Emmanuel Chasseigne}
\address{Emmanuel Chasseigne\newline
    Laboratoire de Math\'ematiques et Physique Th\'eorique (UMR CNRS 7350)\newline
    F\'ed\'eration Denis Poisson (FR CNRS 2964)\newline
    Universit\'e F. Rabelais - Tours\newline
    Parc de Grandmont\newline
    37200 Tours, France}
\email{emmanuel.chasseigne@lmpt.univ-tours.fr}
\urladdr{http://www.lmpt.univ-tours.fr/$\sim$manu}

\begin{abstract}
We give a simplified proof of regularizing effects for  first-order Hamilton-Jacobi Equations of the form $u_t+H(x,t,Du)=0$ in $\R^N\times(0,+\infty)$ in the case where the idea is to first estimate $u_t$. As a consequence, we have a Lipschitz regularity in space and time for coercive Hamiltonians and, for hypo-elliptic Hamiltonians, we also have an H\"older regularizing effect in space following a result of L. C. Evans and M. R. James.
\end{abstract}


   % AMS keywords (used in AMS journals)

%   % acknowledge support, etc
%   \thanks{This research was partially supported by NSF grant
%     DOA-123456789.}
%   \thanks{We would like to thank our colleagues for their helpful
%     criticism.}
\maketitle

%------------------------------------------------
\section{Introduction}\label{sect.intro}
%------------------------------------------------

In this short paper we give a new proof of regularizing effects for
Hamilton-Jacobi Equations 
\begin{equation}\label{HJ}
  u_t+H(x,t,Du)=0 \quad \hbox{in  } \R^N\times(0,\infty)\,,
\end{equation}
in the case when the aim is to estimate $u_t$ first. This new proof is inspired
by ideas introduced in \cite{BS} and then simplified in \cite{BIM}; a precise
comparison between the results and ideas of \cite{BS,BIM} and ours is provided
just after the statement of the main results of this article, at the end of
Section~\ref{amr}. More classical proofs can be found in \cite{Reg} but  with
stronger assumptions and
more tedious proof.

The model equations we have in mind are
\begin{equation}\label{me1}
  u_t+|A(x,t)Du|^m=f(x,t) \quad \hbox{in  } \R^N\times(0,\infty)\,,
\end{equation}
where $f$ is a continuous (typically bounded from below) function and $A$ takes values, in the
set of $N\times N$ symmetric matrices. For such equations, we consider two
cases: the {\em coercive case} for which $A$ is invertible and, as a consequence
$|A(x,t)Du|^m \to + \infty$ as $|p|\to +\infty$, and the {\em non-coercive case}
where $A$ may be degenerate. In both cases, we provide regularizing effects for
{\em bounded from below solutions}. The main improvement in the assumptions is
easy to describe in the coercive case since we just require that $A,f$ are
continuous in $x$ (no uniform continuity assumptions) and, in particular, $f$
may have some growth at infinity. In the non-coercive case, analogous results
hold except that we have to impose far more restrictive assumptions on the
$t$-dependence of the equation.

Of course, for (\ref{me1}), the equation implies that $u_t \leq f(x,t)$ in $\R^N\times(0,\infty)$ and thefore we just need an estimate from below for $u_t$.

To do so, our approach consists in using the exponential transform, $v:=-\exp(-u)$. Notice
that provided $u$ is bounded below (then we can always assume that $u$ is
nonnegative), we get that $v$ is bounded since $-1\leq v\leq 0$. Moreover, $v$
solves a new Hamilton-Jacobi equation 
\begin{equation}\label{eq:HJB.v}
    v_t+ G(x,t,v,Dv)=0\,,\quad \text{with }
     G(x,t,v,p):=-vH\Big(x,t,-\frac{p}{v}\Big)\,.
\end{equation}
In order to estimate $v_t$, a key property (as in all the regularizing effects
proofs) is to have a large enough, positive $G_v$ when $v_t=-G$ is large (but
negative) i.e. when $G$ is large and positive. This leads to an assumption on
$(H_p\cdot p-H)(x,t,p)$ which is classical except that, here, this quantity has
to be large when $H(x,t,p)$ is large, and not when $|p|$ is large as it is
classical for the estimate on~$Du$.

The paper is organized as follows : in Section~\ref{amr}, we state our main
results, one in the case where $H$ is assumed to be coercice, one in the
degenerate case and we give and we deduce full regularizing effects,
\textit{i.e.} in space and time. The proofs of the main theorems are 
are given in Section~\ref{proofs}. Then, in Section~\ref{var}, we treat
several explicit examples. We have put some technical results about Hamiltonian
$G$ in an appendix.

\

\noindent\textit{Acknowledgement --} 
Both authors were partially supported by the french ANR project WKBHJ (Weak KAM beyond Hamilton-Jacobi), ANR-12-BS01-0020.

%%%%%%%%%%%%%%%%%%%%%%%%%%%%%%%%%%%%%%%%%%%%%%%%%%%%%
\section{Assumptions and main results}\label{amr}
%%%%%%%%%%%%%%%%%%%%%%%%%%%%%%%%%%%%%%%%%%%%%%%%%%%%%

%---------------------------------
\subsection{Assumptions}
%---------------------------------

In order to state and prove our results, we use several structure conditions,
which all rely on the following fundamental hypothesis\\

\nhyp{0}
\textit{  The function $H$ is continuous in $\R^N \times [0,T] \times\R^N$ and
    there exists $c_0=c_0(H)\geq0$ such that 
\begin{itemize}
    \item[\textbullet] $H$ is locally
         Lipschitz in the $p$-variable, in a neighborhood of the set
         $\{(x,t,p);\ H(x,t,p)\geq c_0\}$;
         \smallskip
    \item[\textbullet] there exists a continuous, increasing function 
        $\phi:[c_0,\infty[\to[0,\infty[$ such that for some $A>c_0$, 
    \begin{equation}\label{int.cond}
        \int_{A}^{+\infty} \frac{1}{s\phi(s)}\ds < +\infty
    \end{equation}
    and for almost all $x,p\in\R^N$, $t \in [0,T]$,
    $$
  (H_p\cdot p-H)(x,t,p)\geq\phi\big(H(x,t,p)\big)\quad \hbox{a.e. in
  }\{(x,t,p);\ H(x,t,p)\geq c_0\} \,;
  $$
  \item[\textbullet] there exists $\kappa>0$ such that 
    $$|H_p(x,t,p)| \leq \kappa.(H_p\cdot p-H)(x,t,p)\quad 
    \hbox{a.e. in  }\{(x,t,p);\ H(x,t,p)\geq c_0\} \,.$$
\end{itemize}
}

\smallskip

We first point out that, for the model equation (\ref{me1}), \hyp{0}
is satisfied with $c_0=0$ and $\phi(s)=(m-1)s$ by homogeneity, both in
the coercive and non-coercive cases. 

\begin{remark}
Roughly speaking, \hyp{0} is a superlinearity condition on $H$ either in the $p$-variable or a linear function of $p$ (possibly depending on $x,t$ and possibly degenerate). Indeed, it is easy to show that if $H(x,t,p)$ satisfies \hyp{0} then $H(x,t,a(x,t)p)$ also satisfies \hyp{0} for any bounded matrix-valued function $a$. We also refer to Section~\ref{var} for various examples. Even if we were not able to prove it in full generality, Condition~\eqref{int.cond} should be automatically satisfied for convex and superlinear Hamiltonians.
\end{remark}

The first main consequence of \nhyp{0} is
the following: if we set
$$
\forall \tau\in(c_0,\infty)\,,\ F(\tau) 
:=2\int_\tau^\infty\frac{\dsigma}{\sigma\phi(\sigma)} \quad\text{and}\quad
\forall s\in(0,F(c_0))\,,\ \eta(s)=F^{-1}(s)\,,
$$
then $\eta'(s)=-\eta(s)\phi(\eta(s))/2$, 
$s\mapsto\eta(s)$ is decreasing and $\eta(0^+)=+\infty$. Notice that if
$c_0=0$, since $\phi(s)\sim \phi(0)s$ near $s=0$, the integral blows-up and 
$F(0)=\infty$. Similarly, if $\phi(c_0)=0$ we have $F(c_0)=+\infty$.

\

The next assumption concerns the dependence of $H$ in $t$. In order to formulate
it, we notice that $\eta\mapsto 2\eta+\phi(2\eta)$ is increasing from $(c_0,\infty)$
    to $(2c_0+\phi(2c_0),\infty)$. Hence, its inverse, denoted by 
$\psi:(2c_0+\phi(2c_0),\infty)\to(c_0,\infty)$ 
is well-defined and also increasing. Notice that $\psi$ is sublinear but
$\psi(\tau)\to\infty$ as $\tau\to\infty$.

\nhyp{1} \textit{There exists $c_1>0$
    such that the function $H$ is locally Lipschitz in $t$ in a neighborhood of the
set $\{(x,t,p);\ H(x,t,p)\geq c_1\}$ 
and for any $x,p\in\R^N$, $t \in [0,T]$, $$|H_t
(x,t,p)| \leq \psi\big(H(x,t,p)\big)(H_p\cdot p-H)(x,t,p) \quad \hbox{a.e. in  }
\{(x,t,p);\ H(x,t,p)\geq c_1\} \,.$$
}

For the model equation (\ref{me1}), \nhyp{1} is satisfied in the coercive case
as soon as $A$ is Lipschitz continuous in $t$, while, in the non-coercive case,
it imposes rather restrictive (or at least non-natural assumptions) since $A$
should satisfy $|A_t(x,t)p| \leq k |A(x,t)p|.(|A(x,t)p|^m)$ for any $x,t,p$ such
that $|A(x,t)p|^m\geq c_1$, a property which imposes particular types of
degeneracies (see Section~\ref{var} for a study of the scalar case). However
\hyp{1} is satisfied with $c_1=0$ if $A$ does not depend on $t$.

We then define precisely the coercive case: $H$ is coercive if

\nhyp{2} \textit{$H(x,t,p)\to +\infty$ when $|p| \to + \infty$, uniformly for
    $(x,t)$ in a compact subset of $\R^N\times (0,T]$.}\\

And we have to impose some additional restrictions on the behavior of $H$ in $x$
in the non-coercive case \ie when \nhyp{2} does not hold.\\

\nhyp{3} \textit{There exists $c_2>0$ such that the
    function $H$ is locally Lipschitz in $x$ in a neighborhood of the
set $\{(x,t,p);\ H(x,t,p)\geq c_2\}$ and there
exists a constant $\gamma>0$ such that, for any $x,p\in\R^N$, $t \in [0,T]$, 
$$|H_x(x,t,p)|  \leq \gamma (|p|+1)(H_p\cdot p-H)(x,t,p) 
\quad \hbox{a.e. in  }\{(x,t,p);\ H(x,t,p)\geq c_2\} \,.$$}

For the model equation (\ref{me1}), \nhyp{2} is satisfied (in the non-coercive
case) as soon as $A$ is Lipschitz continuous in $x$.

%---------------------------------
\subsection{The main theorem}
%---------------------------------

In order to state our main results, we first introduce 
$$c_*:=\begin{cases}\max(c_0,c_1) & \text{if \hyp{0}-\hyp{1}-\hyp{2} holds}\,,\\
    \max(c_0,c_1,c_2) & \text{if \hyp{0}-\hyp{1}-\hyp{3} holds}\, .
\end{cases}
$$
and we set  
$$t_*=t_*(H):=\begin{cases}\eta^{-1}(c_*)=F(c_*) & \text{if } c_*>0\,,\\
    +\infty & \text{if } c_*=0\, .
\end{cases}
    $$
Notice that this is consistent with the fact that $F(0^+)=+\infty$.

\begin{theorem}\label{thm:main}
We assume that either \hyp{0}-\hyp{1}-\hyp{2} or \hyp{0}-\hyp{1}-\hyp{3}
holds and let $u$ be a continuous, bounded from below solution of \eqref{HJ}.
Then for any $t\in(0,t_*)$ and
$x\in\R^N$, the function $v:=-\exp(-u)$ satisfies
    \begin{equation}\label{ineq:main}
        \inf_{0\leq s\leq t}\big\{ 
            v(x,t)-v(x,s)+(t-s)\eta(s)\,
            \big\}\geq 0\,.
    \end{equation}
\end{theorem}

Notice that since $u$ and $v$ are continuous functions, $u_t$ and $v_t$ are
defined in the sense of distributions. A distributional formulation of
\eqref{ineq:main} is the following: $v_t\geq -\eta(t)$ in $\R^N\times(0,t_*)$.
Similarly, we have $u_t\geq -\eta(t) u$ on the same set.

%---------------------------------
\subsection{Consequences}
%---------------------------------

The general statement of this theorem may hide some non-trivial consequences
concerning global estimates, at least in some particular cases. First, if
$c_*>0$ and we want to get an estimate for $t>t_*$, we can apply the estimate of
the theorem to $w(x,s):=u(x, s+t-t_*/2)$ which is a solution of an analogous
(just shifted in time) pde for $s \in (0,T-t+t_*/2)$. This interval corresponds
to the interval $(t-t_*/2,T)$ for $u$. Hence the
\begin{corollary} If $c_* > 0$, setting $C=\eta(t_*/2)$ we have
$$ v_t \geq - \max(\eta (t), C) \quad \hbox{and} \quad u_t \geq
-\max\big(\eta(t),C\big) u\quad \hbox{a.e.  in }\R^N \times (0,T)\; .$$
\end{corollary}

If $T=+\infty$, the previous estimate gives the lower estimate $u_t\geq -C u$
for $t$ big enough. However, if $c_*=0$ then $t_*=\infty$ and we get a much better
asymptotic estimate
\begin{corollary}\label{cor}
    If $c_*=0$ and $T=+\infty$, 
$$ v_t \geq -\eta (t) \quad \hbox{and} \quad u_t \geq
-\eta(t) u\quad \hbox{a.e.  in }\R^N \times (0,\infty)\; .$$
In particular,  $(u_t)^-\to 0$ locally uniformly in $\R^N$ as $t\to\infty$.
\end{corollary}

We first point out, as we already do it after the statement of \hyp{1}, that
$c_*=0$ implies in particular $c_1=0$, which is a rather restrictive assumption
in the non-coercive case. We come back on this point in Section~\ref{var}.

This asymptotic result leads us to the comparison with the results of
\cite{BS,BIM} which concern cases where $H$ does not depend on $t$. The main
issue in these papers was the asymptotic behavior of $u$ which is expected to be
of ergodic type, namely 
$$ u(x,t) = - \lambda t + \phi(x) + o_t(1)\quad
\hbox{when  } t\to +\infty\; ,$$
where $\lambda$ is the ergodic constant and $\phi$ a solution of the ergodic
problem $$ H(x,D\phi) = \lambda \quad \hbox{in  }\R^N\; .$$ In \cite{BS,BIM},
the key idea is to reduce to the case when $\lambda = 0$ and then to show that
$(u_t)^- \to 0$ (or $(u_t)^+ \to 0$) locally uniformly in $\R^N$ as
$t\to\infty$. To do so, one has to take into account the fact that $\lambda =0$
and $\phi$ plays a key role since it interferes in the analogue of \hyp{0}. With
all these ingredients, the proofs of \cite{BS,BIM} are slightly more
sophisticated that ours. But we point out that, remarking that an analogous
proof could provide regularizing effects is (to the best of our knowledge)
completely new.

This analysis through the ergodic behavior of $u$ also shows that Corollary~\ref{cor} can only hold in particular situations.

%\begin{remark}
%Contemplating this nice asymptotic result, 
%one could try to reduce any situation to the case $c_*=0$ with a simple
%argument, considering $\tilde u(x,t):=u(x,t)+c_*t$ so that 
%$(\tilde u)_t+H(x,t,D\tilde u)-c_*=0$.
%The new Hamiltonian $\tilde H:= H-c_*$ satisfies similar hypotheses to $H$, with
%a constant $c_*(\tilde H)=0$. However, knowing that $(\tilde u^-)_t\to0$ does
%not give a similar good estimate for $(u_t)^-$ of course.
%
%Actually, it seems this difficulty cannot be avoided: 
%if one wants to get some information for small values of 
%$(u_t) ^-=(H)^+$, we need to be able to say something on the small values of
%the Hamiltonian, typically for  $0\leq H\leq c_*$.
%\end{remark}

%-------------------------------------------------
\subsection{Gradient estimates and spatial regularity}

In the \textit{coercive case}, we can deduce a gradient estimate from
\eqref{ineq:main}. Indeed, Theorem~\ref{thm:main} implies
$$ H(x,t,Du) \leq -u_t \leq \eta(t) u\; ,$$
and standard arguments in viscosity solutions' theory show that $u$ is locally
Lipschitz continuous. Moreover, a local estimate of $Du$ follows from the above
inequality.

In the non-coercive case we studied above (Example 3), the bound on $u_t$
implies a local estimate on $|A(x,t)Du|$ and using a result of Evans and
James~\cite{EJ}, we can deduce a local $C^{0,\alpha}$ bound when $A$ satisfies
some ``hypoelliptic'' type assumptions.

%----------------------------------------------------
\section{Proof of Theorem \ref{thm:main}}\label{proofs}
%----------------------------------------------------
In all the proof, we fix $0<t< t_*$ where we recall that $t_*=\eta^{-1}(c_*)$.
Assuming that for some $x_0\in\R^N$, inequality
(\ref{ineq:main}) does not hold, we deduce that for $\alpha,\beta,\epsilon >0$ small
enough with $\epsilon\ll \alpha,\beta$, we have also
\begin{equation}\label{ineq:min}
    \min_{t\geq s\geq 0,\ x,y\in\R^N}\big\{ 
        v(x,t)-v(y,s)+(t-s)\eta(s)+\frac{|x-y|^2}{\epsilon^2}+\alpha|y|^2+\beta\eta(s)\,
            \big\}\leq-\delta\,,
\end{equation}
for some $\delta>0$ which does not depend on $\alpha,\beta,\epsilon$.
Notice first that this min is attained because of the $\alpha$-term, and it
cannot be attained at $s=0$ because of the $\beta$-term. For simplicity, we
still denote by $x,y,s,t$ the variables for which the min is attained. 
Since $v$ is bounded, we have the usual estimates $|x-y|\leq O(\epsilon)$, $|y|=o(\alpha^{-1/2})$.
    
Next, the equation for $v$ given in \eqref{eq:HJB.v} involves a new Hamiltonian
    $G$ whose properties follow from the assumptions on $H$ (see the proof in
    Appendix): if $G(x,t,v,p)\geq c_*$,
    $$\begin{aligned} 
        \text{\hyp{0} implies} \quad
        \text{\gyp{0}}:\ & G_v (x,t,v,p)\geq \phi(G(x,t,v,p))\text{ and }
        |G_p(x,t,v,p)| \leq \kappa\, G_v(x,t,v,p)\,,\\
        \text{\hyp{1} implies}\quad \text{\gyp{1}}:\
        & |G_t (x,t,v,p)|\leq \psi(G(x,t,v,p))\,
        G_v(x,t,v,p)\,,\\
        \text{\hyp{3} implies}\quad \text{\gyp{3}}:\
        & |G_x(x,t,v,p)| \leq \gamma (|p|+1) G_v(x,t,v,p)\,.
    \end{aligned}$$

Using the supersolution inequality for $v(x,t)$ in the viscosity sense, and the
subsolution for $v(y,s)$ we get, with $p_\epsilon=-2(x-y)/\epsilon^2$,
\begin{equation*}
    \begin{aligned}
    -\eta(s) + G\big(x,t,v(x,t),p_\epsilon\big) & \geq0\,,\label{(i)} \\
    -\eta(s) +(t-s+\beta)\eta'(s) + G\big(y,s,v(y,s),p_\epsilon+ 2\alpha y\big)
    & \leq0\,.
\end{aligned}
\end{equation*}

We begin with the {\em coercive} case, assuming \hyp{0}-\hyp{1}-\hyp{2}.
We remark that, if $\alpha$ and $\beta$
remain fixed, then $x,y$ stay in a compact subset of $\R^N$ and $t,s$ in a
compact subset of $(0,t_*)$ and we can let $\epsilon$ tend to zero. By the uniform
continuity of $H$ and $v$ on such compact sets, we end up with
\begin{equation}\label{ineq:visco}
    \begin{aligned}
    -\eta(s) + G\big(x,t,v(x,t),p\big) & \geq0\,, \\
    -\eta(s) +(t-s+\beta)\eta'(s) + G\big(x,s,v(x,s),p+2\alpha x\big) & \leq
    0\,,
\end{aligned}
\end{equation}
for some $p\in\R^N$.
In order to estimate the difference of the hamiltonians, we denote by
$v_1=v(x,t),\ v_2=v(x,s)$, $q_\alpha=2\alpha x$ and
$\xi_\sigma:=(x,t,v_1,p)+\sigma(0,s-t,v_2-v_1,q_\alpha)$. 

The following arguments rely on the next two lemmas, where $s$ and $t$ are
fixed, and $\sigma$ is the running variable. 

\begin{lemma}\label{lem:G.increasing}
    Let $0\leq\sigma_1<\sigma_2\leq1$ such that
    \begin{equation}\label{ineq:claim}
    \eta(s)\leq G(\xi_\sigma) \leq \eta(s)+\phi(\eta(s))\quad
    \text{for any } \sigma \in [\sigma_1,\sigma_2]\,.
\end{equation}
Then, the function $\sigma\mapsto G(\xi_\sigma)$ is increasing on
$[\sigma_1,\sigma_2]$.
\end{lemma}

\begin{proof}
We argue as if $G$ were $C^1$ (otherwise a standard mollification argument
allows to reduce to this case) and compute 
$$ \frac{\mathrm{d}}{\dsigma}\big[G(\xi_\sigma)\big] =G_t(\xi_\sigma)(s-t)+
  G_v(\xi_\sigma)(v_2-v_1)+G_p(\xi_\sigma)\cdot q_\alpha\,.
$$
Inequality \eqref{ineq:min} yields (recall that we have let $\epsilon$
tend to $0$) 
\begin{equation}\label{ineq:v2.v1}
    v_2-v_1\geq
(t-s+\beta)\eta(s)+\alpha|x|^2+\delta\,,
\end{equation}
and therefore, as long as $G(\xi_\sigma) \geq \eta(s)$, we have
$G_v(\xi_\sigma)\geq 0$ and 
$$
  \frac{\mathrm{d}}{\dsigma}\big[G(\xi_\sigma)\big] \geq
  G_t(\xi_\sigma)(s-t)+
  G_v(\xi_\sigma)\big((t-s+\beta)\eta(s)+\alpha|x|^2+\delta\big)
  +G_p(\xi_\sigma)\cdot q_\alpha \,.
$$
Denoting 
$$\begin{aligned}
    I_1 & := G_t(\xi_\sigma)(s-t)+ \frac12 G_v(\xi_\sigma)(t-s)\eta(s)\,,\\
    I_2 & := \frac12 G_v(\xi_\sigma)\alpha|x|^2-G_p(\xi_\sigma)\cdot q_\alpha\,, \\
    I_3 & := \frac12 G_v(\xi_\sigma)\delta\,,
\end{aligned}$$
we have to estimate these three terms, {\em as long as $\eta(s)\leq 
    G(\xi_\sigma)\leq\eta(s)+\phi(\eta(s))$}.
    Recall that since $\eta(s)\geq \eta(t_*)=c_*$, as long as
    $G(\xi_\sigma)\geq\eta(s)$, we have $G(\xi_\sigma)\geq c*$ and we can apply
    the hypotheses on $G$.

    First,
$$
I_1\geq \frac12 (t-s)G_v(\xi_\sigma) \left [ \eta(s)-
    2\frac{|G_t(\xi_\sigma)|}{G_v(\xi_\sigma)} \right ] \,,
$$
and by \gyp{1}, $|G_t(\xi_\sigma)|/G_v(\xi_\sigma)\leq
\psi(G(\xi_\sigma))\leq\psi\big(\eta(s)+\phi(\eta(s))\big)=\eta(s)/2$, the last
equality coming from the definition of $\psi$. Hence $I_1\geq0$ on
$[\sigma_1,\sigma_2]$.

Concerning $I_2$, we have
$$
  I_2\geq \frac12 \alpha G_v(\xi_\sigma)\Big[\, |x|^2 
  -2\frac{\big|G_p(\xi_\sigma)\big|}{G_v(\xi_\sigma)}
  |x|\Big]\geq \frac12 \alpha G_v(\xi_\sigma)\Big[\, |x|^2 
  -2\kappa
  |x|\Big]\,.
$$
Using \gyp{0} we remark that the quantity in the bracket is positive if
$|x|>2\kappa$ and therefore we are left with considering the case when
$|x|\leq 2\kappa$. But in this case, $I_2$
is estimated by $4\alpha\kappa^2 G_v(\xi_\sigma)$ and 
$I_3$ controls this term since $ \delta \geq 4\alpha\kappa^2$
for $\alpha$ small enough, so that $I_3+I_2>0$.

In any case, $I_1+I_2+I_3>0$ on $[\sigma_1,\sigma_2]$ and we deduce that, if
$t_*$ is chosen small enough and $\alpha$ is small enough,
$\frac{\mathrm{d}}{\dsigma}\big[G(\xi_\sigma)\big]>0$ for
$\sigma\in[\sigma_1,\sigma_2]$.
\end{proof}

The next step shows that we can apply Lemma~\ref{lem:G.increasing} on the
whole interval $\sigma_1=0$ and $\sigma_2=1$:
\begin{lemma} For any $\sigma \in [0,1]$
$$
    \eta(s)\leq G(\xi_\sigma) \leq \eta(s)+\phi(\eta(s))\; .
$$
\end{lemma}

\begin{proof}
    We recall that in this proof $s$ is fixed.
We first notice that
viscosity inequalities~\eqref{ineq:visco} yield $G(\xi_0)\geq\eta(s)$ and
$$G(\xi_1)\leq \eta(s)-(t-s+\beta)\eta'(s)=\eta(s)+(t-s+\beta)\eta(s)
\frac{\phi(\eta(s))}{2}\,.$$
From \eqref{ineq:v2.v1} we deduce that $(t-s +\beta)\eta(s)\leq
v_2-v_1-\delta\leq1-\delta$ so
that $G(\xi_1)\leq\eta(s)+(1-\delta)\phi(\eta(s))$.

Now assume by contradiction that $G(\xi_\sigma)>\eta(s)+\phi(\eta(s))$ for some
$\sigma \in [0,1)$ and denote by $\bar \sigma$ the supremum of such $\sigma$.
Then, since $G(\xi_1)\leq \eta(s)+(1-\delta)\phi(\eta(s))$, necessarily 
$\bar \sigma < 1$ and $G(\xi_{\bar \sigma})=\eta(s)+\phi(\eta(s))$. 
Since $G(\xi_1)<G(\xi_{\bar\sigma})=\eta(s)+\phi(\eta(s))$,
by continuity of $\sigma\mapsto
G(\xi_\sigma)$, there exists a subinterval $[\sigma_1,\sigma_2]\subset
[\bar\sigma,1]$, such that $G(\xi_{\sigma_1})>G(\xi_{\sigma_2})$ and for any
$\sigma\in[\sigma_1,\sigma_2]$,
$$\eta(s)\leq G(\xi_{\sigma})\leq \eta(s)+\phi(\eta(s))\,.$$
But this yields a contradiction with Lemma~\ref{lem:G.increasing}. Hence for any
$\sigma\in[0,1]$, $G(\xi_\sigma)\leq \eta(s)+\phi(\eta(s))$.

For the other inequality, the argument is the same: if there exists $\sigma$ such
that $G(\xi_\sigma)<\eta(s)$ (using this time $G(\xi_0)$ and the interval
$[0,\bar \sigma]$ where $\bar \sigma$ is the infimum of the above set of
$\sigma$), we get a similar contradiction with Lemma~\ref{lem:G.increasing}.
Hence the proof is complete.  
\end{proof}

In order to conclude the proof of the first part of Theorem~\ref{thm:main},  we subtract the viscosity inequalities, we obtain
$$
  (t-s+\beta)\eta'(s)+ G(\xi_1)-G(\xi_0)=  
  (t-s+\beta)\eta'(s) + G\big(x,s,v_2,p+q_\alpha\big)-
  G(x,t,v_1,p) \leq 0\,,
$$
and we can perform almost the same computations to compute the left hand side,
except that we write
$$ 
  G_v(\xi_\sigma) = \frac12 G_v(\xi_\sigma) + 
  \frac12 G_v(\xi_\sigma)\geq 
  \frac12 G_v(\xi_\sigma)+\frac12 \phi(\eta(s))\,,
$$
and we obtain
$$\begin{aligned} 
  & (t-s+\beta)\big(\eta'(s)+ \frac12 \phi(\eta(s))\eta(s)\big)  \\
  + & \int_*^1 \Big\{ G_t(\xi_\sigma)(s-t)+ \frac12G_v(\xi_\sigma)
  \big((t-s+\beta)\eta(s)+\alpha|x|^2+\delta\big)
  +G_p(\xi_\sigma)\cdot q_\alpha \Big\}\dsigma \leq 0\,.
  \end{aligned}
$$
By the choice of $\eta(\cdot)$, the first term is $0$ and the integral is
strictly positive for $t_*$ and $\alpha$ small enough.  Therefore we have a
contradiction, proving the claim in the coercive case.

\

Now we turn to the {\em non-coercive case}, assuming \hyp{3} instead of \hyp{2},
where the main difference is that we
are not able to let $\epsilon$ tend to $0$. We define $G(\xi_\sigma)$ in a
similar way and prove that $G(\xi_\sigma) \geq \eta (s)$ (we omit the proof here
since it follows, as in Lemma~\ref{lem:G.increasing}, from the same argument as
below).  Computing the difference of $G$ in the same way, we can write the
difference of the viscosity inequalities as 
$$ 
  \int_*^1 \Big\{ Q_1(\sigma) +
  Q_2(\sigma)+Q_3(\sigma)+Q_4(\sigma)\Big\}\dsigma \leq 0\,,
$$
where
$$
  Q_1(\sigma):= (t-s+\beta)\eta'(s) +
  \frac12G_v(\xi_\sigma)\big((t-s+\beta)\eta(s)+\delta\big)\,,
$$
and by the same arguments as above $Q_1(\sigma) \geq \frac12G_v(\xi_\sigma)\delta$.  
$$
  Q_2(\sigma):= G_x(\xi_\sigma)\cdot(x-y) +
  \frac12G_v(\xi_\sigma)\frac{|x-y|^2}{\epsilon^2}\,,
$$
which can be rewritten as
$$
  Q_2(\sigma)\geq \frac12\epsilon^ 2G_v(\xi_\sigma)\Big[ 
  -\frac{\big|G_x(\xi_\sigma)\big|}{G_v(\xi_\sigma)}|p_\epsilon| +
  \frac{1}{4}|p_\epsilon|^2\Big] \geq \frac12\epsilon^ 2G_v(\xi_\sigma)\Big[ 
  -\gamma |p_\epsilon|(|p_\epsilon|+1) +
  \frac{1}{4}|p_\epsilon|^2\Big] \,.
$$
But $\displaystyle \epsilon^2 |p_\epsilon|^2= \frac{|x-y|^2}{\epsilon^2} = o_\epsilon (1)$ if $\alpha, \beta$ are fixed and therefore $Q_2 \geq o_\epsilon (1)G_v(\xi_\sigma)$ is controlled by $Q_1$ for
$\epsilon$ small enough.
$$
  Q_3(\sigma):=(t-s) \left [ \frac12
  G_v(\xi_\sigma)\eta(s)- G_t(\xi_\sigma) \right ]
$$
is as $I_1$ above and clearly the quantity in the bracket is positive if
$\eta(s)$ is large enough, i.e. $t_*$ small enough.  Finally
$$
  Q_4(\sigma):= \frac12 \alpha G_v(\xi_\sigma)\Big[ |y|^2
  -2\frac{\big|G_p(\xi_\sigma)\big|}{G_v(\xi_\sigma)}|y|\Big]\,,
$$
is treated as $I_2$ above. And the proof is complete.

\

%%%%%%%%%%%%%%%%%%%%%%%%%%%%%%%%%%%%%%%%%%%%%%%%%%%
\section{Examples and applications}\label{var}
\label{sect:var}
%%%%%%%%%%%%%%%%%%%%%%%%%%%%%%%%%%%%%%%%%%%%%%%%%%%

\noindent\textbf{Example 1 ---} 
We first consider the model equation (\ref{me1}) where, for
$m>1$ $$ H(x,t,p) = |A(x,t)p|^m-f(x,t)\,,$$
the function $f$ being bounded from below.
In this case 
$$(H_p\cdot p-H)(x,t,p)= (m-1)|A(x,t)p|^m+f(x,t) = (m-1)H(x,t,p) +mf(x,t)\; .$$
We can assume without loss of generality that $f(x,t)\geq 1$
(which can be obtained by changing $u$ into $u(x,t)+Ct$), so that
\hyp{0} is satisfied with $\phi(s)=(m-1)s$ for $s\geq0$
and $c_0=0$. Moreover, $|H_p(x,t,p)|= m|A(x,t)p|^{m-1}$ and if $H(x,t,p)\geq 0$ 
then $|A(x,t)p|\geq f(x,t)^{1/m}\geq 1$. Therefore 
$$|H_p(x,t,p)| \leq m|A(x,t)p|^m \leq \frac{m}{m-1}(H_p\cdot p-H)(x,t,p)\; .$$
A direct computation shows that $\eta(s)=\big[(m-1)s\big]^{-1}$ which yields
the estimate $u_t\geq -u/\big[(m-1)t\big]$.

\bigskip

\noindent \textbf{Example 2 --- } 
Let us now adress a non-homogeneous situation where, again for $m>1$
\begin{equation}\label{ex2}
 H(x,t,p) = |A(x,t)p|^m-b(x,t)\cdot Du -f(x,t)\,.
\end{equation}
Here, the function $b$ is continuous and there exists $C_b>0$ such that for any
$(x,t)$, $|b(x,t)|\leq C_b$. We are also considering the {\em coercive case}
where $A$ is an invertible matrix and we assume that there exists $C_A>0$ such
that for any $(x,t)$, 
$\|A(x,t)\|\,,\|A^{-1}(x,t)\|\leq C_A$.

In this case, as for the homogeneous case we have
$$(H_p\cdot p-H)(x,t,p)= (m-1)|A(x,t)p|^m+f(x,t) \;.$$
By Young's inequality $$|b(x,t)\cdot p| = |A^{-t}(x,t)b(x,t)\cdot
A(x,t)p|\leq \frac1m |A(x,t)p|^m +
\frac{(m-1)}{m}|{A}^{-t}(x,t)b(x,t)|^{m/(m-1)}$$
where ${A}^{-t}$ is the inverse of the transpose matrix of $A$. 
Therefore, by our hypotheses on $b$ and $A$
there exists two constants $C_1,C_2>0$ such that
$$ (1-\frac1m)|A(x,t)p|^m-C_1 -f(x,t)\leq H(x,t,p) \leq (1+\frac1m)|A(x,t)p|^m+C_2 -f(x,t)\; ,$$
and from there we deduce that \hyp{0} is satisfied with 
$\phi(s)=\big[m(m-1)/(m+1)\big]s-C$ for some constant $C>0$ and for all $s>0$. 

\bigskip

\noindent\textbf{Example 3 ---} 
In the non-coercive case for \eqref{ex2}, \textit{i.e.} if $A$ is not invertible,
the same computations for
\hyp{0} can be done if the Hamiltonian $H$ has the form $$ H(x,t,p) =
|A(x,t)p|^m-c(x,t)\cdot A(x,t)Du -f(x,t)\; ,$$
where $c: \R^N \times [0,T] \to \R^N$. This means that
$b$ has to be of the form $b(x,t)=A^t(x,t)\cdot c(x,t)$. 
Notice that in this degenerate case, at least for large $H$,
\hyp{3} is automatically satisfied if $A$, $c$ and $f$ are Lipschitz continuous in
$x$, uniformly in $t$.
\bigskip

\noindent\textbf{Example 4 ---}
In the non-coercive case it is also interesting to examine the simple example where 
$$H(x,t,p) := a(x,t)|p|^m\; ,$$
where $m>1$ and $a: \R^N \times [0,T] \to [0, +\infty)$ is Lipschitz continuous
in $t$. Taking into account that $\phi(s)=(m-1)s$ and $\psi(s)=s/(2m)$ ,
assumption \hyp{1} imposes 
$$|a_t(x,t)|\leq \frac{m-1}{2m}a(x,t)\big(a(x,t)|p|^{m}\big)$$
for $a(x,t)|p|^m \geq c_1$, for some constant $c_1 \geq 0$. Clearly this implies
$|a_t(x,t)|\leq c\cdot a(x,t)$ for some $c>0$, which implies some particular
degeneracy: for any $x$ either $a(x,t)=0$ for any $t$ or $a(x,t)>0$ for any
$t$. An --admittedly-- not very natural hypothesis. Moreover if we want
$c_1=0$ then necessarily the above constant $c$ has to be $0$ and therefore $a$
cannot depend on $t$.

\bigskip

\noindent\textbf{Remarks ---} 
$(i)$ Concerning \hyp{1}, in all the previous examples above the function $\psi(H)$
behaves like $c.H$ for some constant $c>0$. And in the case of \eqref{ex2},
at least for large $H$, this is
automatically satisfied if $A$, $b$ and $f$ are Lipschitz continuous in $t$, and
uniformly in $x$ in the coercive case. In the non-coercive case, we are in a
similar situation as for Example 4:
we have to impose that $|A_t(x,t)p|\leq C|A(x,t)p|^{m+1}$ for some
constant $C>0$, at least for $|A(x,t)p|$ large enough. And again, this is not a
very natural assumptions. The conclusion is that in non-coercive cases, the
$t$-dependance is a problem in general.

$(ii)$ We can consider other types of growths: if for instance
$H(x,t,p)=e^{|p|}$ we have $\phi(s)=s\ln(s -1)$, defined for $s >2=c_0$.
Hypotheses
\hyp{0}-\hyp{1}-\hyp{2} are satisfied  with $c_1=2$ and in this case,
$\eta(s)\gg s^{-1}$ as $s\to0$.

$(iii)$ If on most examples the function $\phi$ grows at least
linearly, this is not always true. For instance, the case $H(p)=|p|\ln(1+|p|)$
leads to $H_p\cdot p-H \simeq |p|$ for $|p|$ big, which leads to a
function $\phi$ which is not superlinear, only asymptotically linear at infinity.
Indeed, for any $\epsilon >0$, one can check that for $\sigma>0$ big enough,
$\sigma\leq \phi(\sigma)\leq \sigma^{1-\epsilon}$ so that the function
$[\sigma\phi(\sigma)]^{-1}$ is still integrable near infinity.
Hence, \hyp{0} is satisfied and our results apply.

\bigskip

%%%%%%%%%%%%%%%%%%%%%%%%%%%%%%%%%%%%%%%%%%%%%%%%%%%
\section*{Appendix}
%%%%%%%%%%%%%%%%%%%%%%%%%%%%%%%%%%%%%%%%%%%%%%%%%%%

\begin{lemma}\label{lem:H.to.G} 
    Assumptions \hyp{0},\hyp{1},\hyp{3} imply that 
    for any $(x,t,v,p)$ such that $G(x,t,v,p)\geq c_*$,
    $$\begin{aligned} 
        \text{\gyp{0}}\quad & G_v (x,t,v,p)\geq \phi(G(x,t,v,p))\text{ and }
        |G_p(x,t,v,p)| \leq \kappa\, G_v(x,t,v,p)\,,\\
        \text{\gyp{1}}\quad & |G_t (x,t,v,p)|\leq \psi(G(x,t,v,p))\,
        G_v(x,t,v,p)\,,\\
        \text{\gyp{3}}\quad & |G_x(x,t,v,p)| \leq \gamma (|p|+1)G_v(x,t,v,p)\,.
    \end{aligned}$$
\end{lemma}
\begin{proof}
    We first recall that $G(x,t,v,p)=-vH(x,t,-p/v)\leq H(x,t,-p/v)$ since
    $v\in[-1,0)$. So, if $G(x,t,v,p)\geq c_*$, necessarily $H(x,t,-p/v)\geq c_*$
    also.
    
    For \gyp{0}, straightforward computations show that 
    $G_v(x,t,v,p)=(H_p\cdot p-H)(x,t,v,p)$ and $G_p(x,t,v,p)=H_p(x,t,-p/v)$.
    Thus, \hyp{0} implies that if $G(x,t,v,p)\geq c_*$,
    $$G_v(x,t,v,p)\geq\phi(H(x,t,-p/v))\geq\phi(G(x,t,v,p))$$ since $\phi$ is
    increasing. The second part of \gyp{0} is obvious.
    
    To get \gyp{1} we start from \hyp{1}:
    $$|G_t(x,t,v,p)| =|-vH_t(x,t,-p/v)|\leq (-v)\psi(H(x,t,-p/v))G_v(x,t,v,p)\,.$$
    Then we notice that $\psi$ satisfies the following sublinear property: for any
    $\lambda\in[0,1]$, $\lambda\psi(H)\leq \psi(\lambda H)$. Indeed, since
    $\phi$ is increasin, for any $\lambda\in[0,1]$ and $\eta>c_*$,
    $(\lambda2\eta)\phi(\lambda2\eta)\leq\lambda(2\eta\phi(2\eta))$. Taking the
    inverse yields the result. Thus, taking $\lambda=-v\in[0,1]$ 
    we get $|G_t(x,t,v,p)|\leq\psi(G(x,t,v,p))G_v(x,t,v,p)$. 

    The proof of \gyp{3} is similar:
    $|G_x(x,t,v,p)|\leq|v||H_x(x,t,-p/v)|$. Then, using \hyp{3},
    $|v|\leq 1$, we get
    $$|G_x(x,t,v,p)|\leq \gamma\left(\Big|\frac{p}{v}\Big|+1\right)|v| G_v(x,t,v,p)\\
      \leq \gamma (|p|+1) G_v(x,t,v,p)\,,$$
which gives the result.
\end{proof}


\begin{thebibliography}{00}
% please try to use the bibitem system -
% the references should be in alphabetical order of authors' names.
% Articles with a single author first, author will 1 co-author next,
% then author with several co-authors;


% \bibitem{label}
% Text of bibliographic item

\bibitem{Reg} Barles, Guy, \textit{Regularity results for first-order Hamilton-Jacobi Equations}. J.
Differential and Integral Equations, vol 3, n${}^\circ$ 1,(1990), pp~103-125. 

\bibitem{BIM} Barles, Guy; Ishii, Hitoshi; Mitake, Hiroyoshi, \textit{A new PDE approach
    to the large time asymptotics of solutions of Hamilton-Jacobi equations.}
Bull. Math. Sci. 3 (2013), no. 3, 363--388.

\bibitem{BS} Barles, Guy; Souganidis, Panagiotis E,
On the Large Time Behavior of Solutions 
of Hamilton-Jacobi Equations.  SIAM J. Math. Anal. 31, No.4, 925-939 (2000). 

\bibitem{EJ} Evans, L. C.; James, M. R.,
{The Hamilton-Jacobi-Bellman equation for time-optimal control},
{SIAM J. Control Optim.}, {27}, {(1989)}, {6}, {1477--1489}.

\end{thebibliography}
\end{document}